\documentclass[11pt]{article}
\usepackage{amsmath}
\usepackage{float}
\usepackage{amsthm}
\usepackage{amssymb}
\usepackage[utf8x]{inputenc}
\usepackage{caption}
\usepackage{graphics}
\usepackage{enumerate} 
\usepackage{algorithmic}
\usepackage[Algorithm,ruled]{algorithm}
\usepackage{titlesec}
\usepackage{sectsty}
\titleformat*{\section}{\LARGE\bfseries}
\titleformat*{\subsection}{\Large\bfseries}
\titleformat*{\subsubsection}{\large\bfseries}
\sectionfont{\Huge}
\subsectionfont{\large}
\subsubsectionfont{\normalsize}
\usepackage{multicol}
\usepackage{lmodern}
\usepackage{marvosym} 
\usepackage{lipsum}
\usepackage{mwe}
\usepackage{caption}
\usepackage{subcaption} 
\newtheoremstyle{case}{}{}{}{}{}{:}{ }{}
\theoremstyle{case}

\newcommand{\be}{\begin{equation}}
\newcommand{\ee}{\end{equation}}
\newcommand{\ben}{\begin{eqnarray*}}
\newcommand{\een}{\end{eqnarray*}}
\newtheorem{examp}{\sc example}
\newtheorem{remk}{\sc remark}
\newtheorem{corol}{\sc corollary}
\newtheorem{lemma}{\sc lemma}
\newtheorem{theorem}{\sc theorem}
\newtheorem{defn}{\sc definition}
\newcommand{\bt}{\begin{theorem}}
\newcommand{\et}{\end{theorem}}
\newcommand{\bl}{\begin{lemma}}
\newcommand{\el}{\end{lemma}}
\newcommand{\bed}{\begin{defn}}
\newcommand{\eed}{\end{defn}}
\newcommand{\brem}{\begin{remk}}
\newcommand{\erem}{\end{remk}}
\newcommand{\bex}{\begin{examp}}
\newcommand{\eex}{\end{examp}}
\newcommand{\bcl}{\begin{corol}}
\newcommand{\ecl}{\end{corol}}

\newcommand{\NI}{\noindent}

\theoremstyle{definition}
\theoremstyle{remark}

\numberwithin{equation}{section}
\numberwithin{theorem}{section}
\numberwithin{lemma}{section}

\begin{document}

\title{\large\bf\sc On the solution set of semi-infinite tensor complementarity problem}

\author{R. Deb$^{a,1}$ and A. K. Das$^{b,2}$\\
\emph{\small $^{a}$Jadavpur University, Kolkata , 700 032, India.}\\	
\emph{\small $^{b}$Indian Statistical Institute, 203 B. T.
	Road, Kolkata, 700 108, India.}\\
\emph{\small $^{1}$Email: rony.knc.ju@gmail.com}\\
%\emph{\small $^{2}$Email: aritradutta001@gmail.com}\\
\emph{\small $^{2}$Email: akdas@isical.ac.in}\\
}

\date{}

\maketitle

\begin{abstract}
\NI In this paper, we introduce semi-infinite tensor complementarity problem to provide an approach for considering a more realistic situation of the problem. We prove the necessary and sufficient conditions for the existence of the solution set. In this context, we study the error bounds of the solution set in terms of residual function.\\

\noindent{\bf Keywords:} Semi-infinite tensor complementarity problem, error bound, residual function, $R_0$-tensor.\\

\noindent{\bf AMS subject classifications:} 90C30, 90C33, 15A69. 
\end{abstract}
\footnotetext[1]{Corresponding author}

\section{Introduction}
The tensor complementarity problem, TCP$(q,\mathcal{A})$, was introduced by Song and Qi \cite{song2017properties} and \cite{song2015properties}. By reformulating the multilinear game as a tensor complementarity problem, Huang and Qi \cite{huang2017formulating} established a bridge between the multilinear game and tensor complementarity problem. They demonstrated that finding a Nash equilibrium point of the multilinear game is equivalent to finding a solution to the resulting TCP.

\NI Let $\mathcal{A}$ be a tensor of order $m$ and dimension $n,$ i.e., $\mathcal{A}\in T_{m,n}$ and a vector $q\in \mathbb{R}^n,$ the tensor complementarity problem denoted by TCP$(q,\mathcal{A})$ is to find $x \in \mathbb{R}^n$ such that, \begin{equation}\label{ Tensor Complementarity problem}
		q + \mathcal{A}x^{m-1}\geq 0, \;\;\;\;  x\geq 0, \;\;\;\; \mbox{and}\;\;\;\; x^{T}(q + \mathcal{A}x^{m-1})=0.
\end{equation}
When the order of the tensor $m=2$ then the problem reduces to a linear complementarity problem. Let $A$ be an $n\times n$ real matrix and a vector $q\in \mathbb{R}^n,$ the linear complementarity problem, denoted by LCP$(q,A)$ is finding $x \in \mathbb{R}^n$ such that
\begin{equation}\label{linear comp problem}
		q + A x \geq 0, \;\;\;\; x\geq 0, \;\;\;\; \mbox{and}\;\;\;\; x^{T}(q+Ax)=0.
	\end{equation}

\NI The idea of complementarity considers a large number of optimization problems. The problems which can be constituted as linear complementarity problem includes linear programming, linear fractional programming, convex quadratic programming and the bimatrix game problem. It is well considered in the literature on mathematical programming and occurs in a number of applications in operations research, control theory, mathematical economics, geometry and engineering. For recent works on this problem and applications see \cite{mohan2001more}, \cite{mohan2001classes}, \cite{neogy2006some}, \cite{neogy2005almost}, \cite{mohan2004note}, \cite{das2018invex}, \cite{dutta2023some}, \cite{jana2021iterative}, \cite{jana2018semimonotone}, \cite{neogy2008mixture} and \cite{neogy2016optimization} references cited therein.

\NI The concept of PPT is originally motivated by the well-known linear complementarity problem, and applied in many other settings. The PPT is basically a transformation of the matrix of a linear system for exchanging unknowns with the corresponding entries of the right hand side of the system. For details see \cite{das2017finiteness}, \cite{mondal2016discounted}, \cite{neogy2012generalized}, \cite{das2016properties} and \cite{neogy2005principal}.

\NI Due to their prominence in scientific computing, complexity theory, and the theoretical underpinnings of linear complementarity problems, a number of matrix classes and their subclasses have received substantial study. For recent work on this problem and applications see \cite{jana2019hidden}, \cite{das2016generalized}, \cite{dutta2022on}, \cite{neogy2011singular}, \cite{neogy2009modeling}, \cite{das2018some}, \cite{jana2018processability}. For multivariate analysis and game problem, See \cite{mondal2016discounted}, \cite{jana2021more}, \cite{jana2018processability}, \cite{neogy2013weak}, \cite{neogy2008mathematical}, \cite{neogy2005linear} and references cited therein.

An implicit assumption shared by the nonlinear complementarity problem is that the information about the mapping $F$ and the cone involved are all fixed and completely independent of other related parameters. However, this type of formulation is unable to model all realistic situations of the problem and fails to explain the complete reality. For example, in optimal control or engineering design fields \cite{chen2005optimal}, the data of the problem involves a time parameter; in non-cooperative games (e.g., generalized Nash equilibrium \cite{facchinei2009generalized}), the strategy of each player is dependent on the strategy of the other players in case of the realistic model. Here we introduce semi-infinite tensor complementarity (SITCP), a generalized version of nonlinear omplementarity to accommodate more number of realistic situations into the model.

The paper is organised as follows. Section 2 presents some basic notations and results which are used in the next section. In section 3 we prove the existence of the solution set for semi-infinite tensor complementarity problem with some assumptions. We establish a connection between the solution set of semi-infinite tensor complementarity problem and the solution sets of its equivalent tensor complementarity problems. We establish the necessary and sufficient conditions for the error bound of solution set to be level bounded in terms of residual function.

\section{Preliminaries}
We begin by outlining the fundamental concepts and the notation that will be applied throughout the text. Here we consider the vectors, matrices and tensors of real entries. For any positive integer $n,$ the set $\{ 1,...,n \}$ is denoted by $[n]$ . Let $\mathbb{R}^n$ denote the $n$-dimensional Euclidean space and $\mathbb{R}^n_+ =\{ x\in \mathbb{R}^n : x\geq 0 \}.$ Any vector $x\in \mathbb{R}^n$ is a column vector unless specified otherwise. The Euclidean norm of a vector $x$ is defined as $\|x\|_2 = \sqrt{|x_1^2| + \cdots |x_n^2|}.$ The distance of $x$ from $A\subseteq \mathbb{R}^n$ is denoted as dist$(x;A)$ and is defined as dist$(x;A) =$ inf$\{(x,a) : a \in A\}.$ The diameter of a set is denoted as diam$A$ and is defined as diam$A=\sup_{x,y \in A} \|x-y\|.$ The unit ball in $\mathbb{R}^n$ is $\mathbb{B}=\{ x\in \mathbb{R}^n : \|x\| \leq 1 \}.$ An $m$th order $n$ dimensional real tensor $\mathcal{A}= (a_{i_1 ... i_m}) $ is a multidimensional array of entries $a_{i_1 ... i_m} \in \mathbb{R}$ where $i_j \in [n]$ with $j\in [m]$. The set of all $m$th order $n$ dimensional real tensors are denoted by $T_{m,n}.$ Shao \cite{shao2013general} introduced a product of tensors. Let $\mathcal{A}$ with order $q \geq 2$ and $\mathcal{B}$ with order $k \geq 1$ be two $n$-dimensional tensors. The product of $\mathcal{A}$ and $\mathcal{B}$ is a tensor $\mathcal{C}$ of order $(q-1)(k-1) + 1$ and dimension $n$ with entries $c_{j \alpha_1 \cdots \alpha_{m-1} } =\sum_{j_2, \cdots ,j_m \in [n]} a_{j j_2 \cdots j_m} b_{j_2 \alpha_1} \cdots b_{j_m \alpha_{m-1}},$ where $j \in [n] $, $\alpha_1, \cdots, \alpha_{m-1} \in [n]^{k-1}.$

\NI Then for a tensor $\mathcal{A}\in T_{m,n}$ and $x\in \mathbb{R}^n,\; \mathcal{A}x^{m-1}\in \mathbb{R}^n $ is a vector defined by
	\[ (\mathcal{A} x^{m-1})_i = \sum_{i_2, ...,i_m =1}^{n} a_{i i_2 ...i_m} x_{i_2} \cdots x_{i_m} , \;\forall \; i \in [n], \]
	and $\mathcal{A}x^m\in \mathbb{R} $ is a scalar defined by
 \begin{equation*}
     x^T \mathcal{A}x^{m-1} = \mathcal{A}x^m = \sum_{i_1,...,i_m =1}^{n} a_{i_1  ...i_m} x_{i_1}  \cdots x_{i_m}.
 \end{equation*}	

\noindent Given a vector $q \in \mathbb{R}^n$ and a tensor $\mathcal{A} \in T_{m,n}$ the set of feasible solution of TCP$(q,\mathcal{A})$ is defined as FEA$(q,\mathcal{A})= \{x\in \mathbb{R}^n_+ : q + \mathcal{A}x^{m-1} \geq 0\}$ and the solution set of TCP$(q,\mathcal{A})$ as $S=$ SOL$(q,\mathcal{A})= \{x\in$ FEA$(q,\mathcal{A}) :  x^{T}(q + \mathcal{A}x^{m-1})= 0\}.$ 
A residual function $r(x)$ is a global (local) error bound for TCP if $\exists$ some constant $c > 0$ (and $\epsilon > 0$) such that for each $x \in \mathbb{R}^n$ (when $r(x) \leq \epsilon$)
\begin{equation}
    \mbox{dist}(x,S) \leq c r(x).
\end{equation}
 
\NI We consider some definitions and results which are required for the next sections.

\begin{defn}\cite{shao2016some}, \cite{deb2023more}
The $i$th row subtensor of $\mathcal{A}= (a_{i_1 ... i_m}) \in T_{m,n}$ is denoted by $R_i(\mathcal{A})$ and its entries are given as $(R_i(\mathcal{A}))_{i_2 ... i_m}=(a_{i i_2... i_m})$, where $i_j\in [n]$ and $2\leq j\leq m.$
\end{defn}

\begin{defn}\cite{wets1998variational}
A function $f: \mathbb{R}^n \mapsto \mathbb{R}^n$ is said to be level bounded if for every $\alpha \geq 0$ the level set $\{x\in \mathbb{R}^n: f(x) \leq \alpha \}$ is bounded.
\end{defn}

\begin{defn}\cite{song2015properties}
    A tensor $\mathcal{A}\in T_{m,n} $ is said to be a $S$-tensor if the system
    \[ \mathcal{A}x^{m-1} >0, \;\;\;\; x > 0  \] has a solution.
\end{defn}

\begin{defn}\cite{song2015properties}
A tensor $\mathcal{A} \in T_{m,n} $ is said to be a $P$-tensor, if for each $x\in \mathbb{R}^n \backslash \{0\}$, there exists an index $i\in [n]$ such that $_i \neq 0$ and $x_i (\mathcal{A}x^{m-1})_i > 0$.
\end{defn}

\begin{defn}\cite{song2016properties}, \cite{song2015properties}
A tensor $\mathcal{A}\in T_{m,n} $ is said to be a $R_0$-tensor if the TCP$(0, \mathcal{A})$ has unique zero solution. 
\end{defn}

\begin{defn}\cite{zhou2013solution}
    For a given $\epsilon \geq 0$ a residual function $r(x)$ is said to be an $\epsilon$-error bound for TCP if $\exists\, c>0$ such that dist$(x, S) \leq c r(x) + \epsilon\;\;\;\; \forall \;x \in \mathbb{R}^n.$ If $\epsilon = 0,$ the definition reduces to the error bound.
\end{defn}

\begin{theorem}\label{FIP}\cite{munkrestopology}
A non-empty family $A$ of subsets of $\mathbb{R}^n$ is said to have the finite intersection property (FIP) if the intersection over any finite subcollection of $A$ is non-empty.
\end{theorem}

\begin{theorem}\cite{bai2016global}
For a $P$-tensor $\mathcal{A} \in T_{m,n}$ and any $q\in \mathbb{R}^n$ the solution set of TCP$(q,\mathcal{A})$ is nonempty and compact.
\end{theorem}

\begin{theorem}\cite{song2016properties}\label{boundedness theorem for R_0-tensor}
    If $\mathcal{A}\in T_{m,n}$ is a an $R_0$-tensor then for $q \in \mathbb{R}^n,$ the solution set of the TCP$(q, \mathcal{A})$ is bounded.
\end{theorem}

%\begin{theorem}\label{first proof of error bound}\cite{zheng2019global}
%Given $\tilde{q}\in \mathbb{R}^n,\; \mathcal{M}\in T_{l,n}$ with $\mathcal{M}$ being a $P$-tensor and $\alpha(F_{\mathcal{M}})$ is defined by (\ref{defn of alphaFA}). Let $z$ be a solution of TCP$(\tilde{q},\mathcal{M}).$ Suppose for any $u\in \mathbb{R}^n,$ the residue function $\tilde{v}$ is defined as $\tilde{v}=\min \{ u, [(\mathcal{M}(u-z)^{l-1})^{[\frac{1}{l-1}]} + (\mathcal{M}z^{l-1} + \tilde{q})^{[\frac{1}{l-1}]} ]\}.$ Then for any $u\in\mathbb{R}^n,$
%\begin{equation}\label{first equation of error bound}
%    \frac{1}{1+ \|\mathcal{M}\|^{\frac{1}{l-1}}}_{\infty} \|\tilde{v}\|_{\infty} \leq 
%    \|u-z\|_{\infty} 
%    \leq \frac{1+\|\mathcal{M}\|^{\frac{1}{l-1}}_{\infty}}{\alpha(F_\mathcal{M})} \|\tilde{v}\|_{\infty}
%\end{equation}
%\end{theorem}

\section{Main results}

We begin by introducing semi-infinite tensor complementarity problem SITCP$( q(\omega), \mathcal{A}(\omega), \Omega)$.\\

\NI Find a vector $x\in \mathbb{R}^n$ such that
\begin{equation}\label{SICP}
    x \geq 0,\;\; F(x,\omega) \geq 0, \;\; x^T F(x,\omega) = 0,\;\; \omega \in \Omega
\end{equation}
where $F:\mathbb{R}^n \times \Omega \mapsto \mathbb{R}^n,$ $F= \mathcal{A}(\omega)x^{m-1} + q(\omega)$ and $\Omega$ is a set in $\mathbb{R}^p.$ The solution set of SITCP$(q(\omega), \mathcal{A}(\omega), \Omega)$ is denoted by $S^*=$ SOL$(q(\omega), \mathcal{A}(\omega), \Omega)= \{x\geq0: \mathcal{A}(\omega)x^{m-1}+ q(\omega) \geq 0,\;  x^{T}(\mathcal{A}(\omega)x^{m-1}+ q(\omega))= 0, \; \forall\; \omega \in \Omega\}.$\\

Here we establish the necessary and sufficient conditions for $S^*$ to be non-empty.

\begin{theorem}
Consider the SITCP$(q(\omega), \mathcal{A}(\omega), \Omega)$ and% Then $S^*=$ $\cap_{\omega \in \Omega}$ SOL$(q(\omega), \mathcal{A}(\omega)).$
 $\mathcal{A}(\omega_0)$ is an $R_0$-tensor for some $\omega_0 \in \omega.$ $S^*\neq \phi$ if and only if $\cap_{i=1}^p \mbox{SOL}(q(\omega_i), \mathcal{A}(\omega_i)) \neq \phi$ for finitely many points $\omega_1, ... , \omega_p \in \Omega.$
\end{theorem}
\begin{proof}
    If part: Since $S^* \subseteq \cap_{i=1}^p$ SOL$(q(\omega_i), \mathcal{A}(\omega_i)),$ $S^* \neq \phi$ implies that $\cap_{i=1}^p$ SOL$(q(\omega_i)$, $\mathcal{A}(\omega_i)) \neq \phi.$

    Only if part: Since $\mathcal{A}(\omega_0)$ is an $R_0$-tensor, by Theorem \ref{boundedness theorem for R_0-tensor}, the set SOL$(q(\omega_0), \mathcal{A}(\omega_0))$ is bounded. Which in turn implies the boundedness of $S^*.$ On the other hand, since SOL$(q(\omega), \mathcal{A}(\omega))$ is closed for each $\omega$, so is $S^*.$ Thus $S^*$ is compact. By Theorem \ref{FIP} of finite intersection of compact sets, we obtain the result.
\end{proof}

In the next result we find the position of $S^*.$ For this purpose we consider the followings:\\

\NI (i) $(\mathcal{A}_{\max})_{i_1 \cdots i_m} = (\Bar{a}_{i_1 \cdots i_m})= \max_{\omega \in \Omega} a_{i_1 \cdots i_m}(\omega)$\\
(ii) $(\mathcal{A}_{\min})_{i_1 \cdots i_m} = (a^{\prime}_{i_1 \cdots i_m})= \min_{\omega \in \Omega} a_{i_1 \cdots i_m}(\omega)$\\
(iii) $(q_{\max})i = \max_{\omega \in \Omega} q_i(\omega)$\\
(iv) $(q_{\min})i = \min_{\omega \in \Omega} q_i(\omega)$

\begin{theorem}\label{relation with max and min}
    Consider SITCP$(q(\omega), \mathcal{A}(\omega), \Omega).$ If $\Omega$ is compact and $\mathcal{A}(\omega)$ and $q(\omega)$ are continuous on $\Omega,$ then $S^* \subseteq \mbox{SOL}(q_{\max}, M_{\max}) \cap \mbox{SOL}(q_{\min}, M_{\min}).$ Furthermore, suppose in each row sub-tensor, $R_i(\mathcal{A}(\omega))$ and in each row of $q(\omega), \; q_i(\omega)$ the minimum (and maximum) is attained by a common $\omega^{\prime}$ (and $\Bar{\omega}$), i.e., for each $i = 1, 2,..., n,$ there exist $\omega^{\prime}_i,\; \Bar{\omega}_i \; \in  \Omega$ such that $R_i(\mathcal{A}_{\min}) = R_i(\mathcal{A}(\omega^{\prime}_i))$ , $(q_{\min})_i =q(\omega^{\prime}_i)_i$ and $R_i(\mathcal{A}_{\max}) = R_i(\mathcal{A}(\Bar{\omega}_i)),$ $(q_{\max})_i =q(\Bar{\omega}_i)_i.$ Then 
    \[ S^* = \mbox{SOL}(q_{\max}, \mathcal{A}_{\max}) \cap \mbox{SOL}(q_{\min}, \mathcal{A}_{\min}). \]
\end{theorem}
\begin{proof}
     By the rules of maximization and minimization for the summation of functions in Exercise 1.36 of \cite{wets1998variational}, we have, 
     \begin{equation}\label{Inequation 1}
         \max_{\omega \in \Omega} (\mathcal{A}(\omega)x + q(\omega)) \geq  \mathcal{A}_{\max} x + q_{\max} 
     \end{equation}
     \begin{equation}\label{Inequation 2}
         \min_{\omega \in \Omega} (\mathcal{A}(\omega)x + q(\omega)) \leq  \mathcal{A}_{\min} x + q_{\min}
     \end{equation}
for all $x \geq 0.$ By using the inequalities (\ref{Inequation 1}) and (\ref{Inequation 2}) and following an argument similar to that for Theorem 2.2 of \cite{zhou2013solution}, we obtain
\begin{equation}\label{inequation for man min and S^* 1}
   S^* \subseteq \mbox{SOL}(q_{\max}, \mathcal{A}_{\max}) \cap \mbox{SOL}(q_{\min}, \mathcal{A}_{\min}). 
\end{equation}

Now to prove the second part let us assume that the conditions of Theorem \ref{relation with max and min} hold. Then $(\mathcal{A}_{\min}x^{m-1})_i = (\mathcal{A}(\omega^{\prime})x^{m-1})_i$ and $(\mathcal{A}_{\min}x^{m-1})_i + (q_{\min})_i = (\mathcal{A}(\omega^{\prime}_i)x^{m-1} +q(\omega_i))_i.$ Therefore SOL$(q_{\min},\mathcal{A}_{\min})$ =SOL$(q(\omega^{\prime}_i), \mathcal{A}(\omega^{\prime}_i)).$ Similarly we have SOL $(q_{\max},$ $\mathcal{A}_{\max})$ =SOL $(q(\bar{\omega}_i),\mathcal{A}(\bar{\omega}_i)).$ Since $S^* = \cap_{\omega \in \Omega} \mbox{SOL}(q(\omega), \mathcal{A}(\omega)),$ using (\ref{Inequation 1}) and (\ref{Inequation 2})
we conclude that
\[ S^* = \mbox{SOL}(q_{\max}, \mathcal{A}_{\max}) \cap \mbox{SOL}(q_{\min}, \mathcal{A}_{\min}). \]
\end{proof}

Here we provide an example to illustrate the result of Theorem \ref{relation with max and min}.

\begin{examp}
    Let $\mathcal{A}(\omega) \in T_{3,2}$ and $q(\omega) \in \mathbb{R}^2$ be such that $a_{111}=(1-2\omega^3), \; a_{121}= 1-\omega, \; a_{112}=(1-\omega),\; a_{122}=-1,\; a_{211}= a_{212}= a_{221}=0\; a_{222}=-\omega^2\;$ and $q(\omega)=\left(\begin{array}{c}
        1 \\
        \omega^2 
    \end{array}\right)$ and $\omega \in \Omega = [0,1].$ Then $\mathcal{A}(\omega)x^2=\left( \begin{array}{c}  (1-2\omega^3)x_1^2 + 2 (1-\omega) x_1 x_2 -x_2^2 \\ -\omega^2 x_2^2 \end{array} \right).$ Now consider the SITCP$( q(\omega), \mathcal{A}(\omega), \Omega)$ which is to find $x=\left( \begin{array}{c}  x_1 \\ x_2 \end{array} \right) \in \mathbb{R}^2$ such that
\begin{equation}\label{example main 1}
    x_1 \geq 0;\;\;\;\;  (\mathcal{A}(\omega)x^2)_1 \geq 0;\;\;\;\;  x_1[(1-2\omega^3)x_1^2 + 2 (1-\omega) x_1 x_2 -x_2^2 + 1] =0,
\end{equation}
\begin{equation}\label{example main 2}
    x_2 \geq 0;\;\;\;\; (\mathcal{A}(\omega)x^2)_2 \geq 0;\;\;\;\; x_2[-\omega^2 x_2^2+ \omega^2] =0.
\end{equation}
Solving the equations (\ref{example main 1}) and (\ref{example main 2}) we obtain the solution set for $\omega \in \Omega,$ which is $\left\{ \left( \begin{array}{c}  0 \\ 0 \end{array} \right),\; \left( \begin{array}{c}  0 \\ 1 \end{array} \right),\; \left( \begin{array}{c}  \frac{2(1-\omega)}{1-2\omega^3} \\ 1 \end{array} \right) \right\}.$ For $\omega=1$ we get $\frac{2(1-\omega)}{1-2\omega^3}=0$ so $S^* =\left\{ \left( \begin{array}{c}  0 \\ 0 \end{array} \right),\; \left( \begin{array}{c}  0 \\ 1 \end{array} \right) \right\}.$

\NI Now for the given tensor $\mathcal{A}(\omega),$ let $\mathcal{A}_{\max} = (\Bar{a}_{ijk})\in T_{3,2}.$ Then we obtain $\Bar{a}_{111}=1, \; \Bar{a}_{121}=\Bar{a}_{112}=1,\; \Bar{a}_{122}=-1$ and $\Bar{a}_{211}= \Bar{a}_{212}= \Bar{a}_{221}= \Bar{a}_{222}=0\;$ and $q_{max}=\left(\begin{array}{c}
        1 \\
        1 
    \end{array}\right).$ Then the TCP$(q_{\max}, \mathcal{A}_{\max})$ which is finding $x=\left( \begin{array}{c}  x_1 \\ x_2 \end{array} \right) \in \mathbb{R}^2$ such that
\begin{equation}\label{example max 1.1}
    x_1 \geq 0;\;\;\;\; (\mathcal{A}_{\max}x^2)_1 \geq 0;\;\;\;\; x_1[x_1^2 + 2 x_1 x_2 -x_2^2 + 1] =0,
\end{equation}
\begin{equation}\label{example max 2.1}
    x_2 \geq 0;\;\;\;\; (\mathcal{A}_{max}x^2)_2 \geq 0;\;\;\;\; x_2[0+1] =0.
\end{equation}
Solving (\ref{example max 1.1}) and (\ref{example max 2.1}) we get SOL$(q_{\max}, \mathcal{A}_{\max}) = \left\{ \left( \begin{array}{c}  0 \\ 0 \end{array} \right)\right\}.$

\NI Again, let $\mathcal{A}_{\min} = (a^{\prime}_{ijk})\in T_{3,2}.$ Then we obtain $a^{\prime}_{111}=-1, \; a^{\prime}_{121}=a^{\prime}_{112}=0,\; a^{\prime}_{122}=-1$ and $a^{\prime}_{211}= a^{\prime}_{212}= a^{\prime}_{221}= 0,\; a^{\prime}_{222}=-1\;$ and $q_{min}=\left(\begin{array}{c}
        1 \\
        0 
    \end{array}\right).$ Then TCP$( q_{\min}, \mathcal{A}_{\min})$ is finding $x=\left( \begin{array}{c}  x_1 \\ x_2 \end{array} \right) \in \mathbb{R}^2$ such that
\begin{equation}\label{example min 1.1}
    x_1 \geq 0;\;\;\;\; (\mathcal{A}_{\min}x^2)_1 \geq 0;\;\;\;\; x_1[-x_1^2 + -x_2^2 + 1] =0,
\end{equation}
\begin{equation}\label{example min 2.1}
    x_2 \geq 0;\;\;\;\; (\mathcal{A}_{min}x^2)_2 \geq 0;\;\;\;\; x_2[-x_2^2] =0.
\end{equation}
Solving (\ref{example min 1.1}) and (\ref{example min 2.1}) we have SOL$(q_{\min}, \mathcal{A}_{\min}) = \left\{ \left( \begin{array}{c}  0 \\ 0 \end{array} \right), \; \left( \begin{array}{c}  1 \\ 0 \end{array} \right)\right\}.$

 \NI Thus $S^* \supset$ SOL$(q_{\max}, \mathcal{A}_{\max}) \cap$ SOL$(q_{\min}, \mathcal{A}_{\min}),$i.e., the inclusion is strict.

 \NI Now we replace $q(\omega)$ by $\Bar{q}= \left( \begin{array}{c} 1 \\ 1 \end{array} \right).$ Then the SITCP$(\Bar{q}, \mathcal{A}(\omega), \Omega)$ which is to find $x=\left( \begin{array}{c}  x_1 \\ x_2 \end{array} \right) \in \mathbb{R}^2$ such that
\begin{equation}\label{example main 1.1}
    x_1 \geq 0;\;\;\;\;  (\mathcal{A}(\omega)x^2)_1 \geq 0;\;\;\;\;  x_1[(1-2\omega^3)x_1^2 + 2 (1-\omega) x_1 x_2 - x_2^2 + 1] =0,
\end{equation}
\begin{equation}\label{example main 2.1}
    x_2 \geq 0;\;\;\;\; (\mathcal{A}(\omega)x^2)_2 \geq 0;\;\;\;\; x_2[-\omega^2 x_2^2+ 1] =0.
\end{equation}
Solving the equations (\ref{example main 1.1}) and (\ref{example main 2.1}) we obtain the solution set for $\omega \in \Omega,$ which is $\left\{ \left( \begin{array}{c}  0 \\ 0 \end{array} \right),\; \left( \begin{array}{c} \alpha (\omega) \\ \frac{1}{\omega} \end{array} \right) \right\}.$ Here $\alpha(\omega)$ is the positive root of the equation
\begin{equation*}
    \omega^2 (1-2\omega^3) x_1^2 -2\omega (1-\omega) x_1 -(1-\omega^2) = 0.
\end{equation*}
For $\omega=1$ we get $\alpha(1)=0.$ Therefore $S^* =\left\{ \left( \begin{array}{c}  0 \\ 0 \end{array} \right) \right\}.$

\NI Then TCP$(q_{\max}, \mathcal{A}_{\max}) =$ TCP$(\Bar{q}, \mathcal{A}_{\max}).$ SOL$(\Bar{q}, \mathcal{A}_{\max})=\left\{ \left( \begin{array}{c}  0 \\ 0 \end{array} \right) \right\}.$

\NI Now, TCP$(\Bar{q}, \mathcal{A}_{\min})$ is finding $x=\left( \begin{array}{c}  x_1 \\ x_2 \end{array} \right) \in \mathbb{R}^2$ such that
\begin{equation}\label{example min 1.2}
    x_1 \geq 0;\;\;\;\; (\mathcal{A}_{\min}x^2)_1 \geq 0;\;\;\;\; x_1[-x_1^2 + -x_2^2 + 1] =0,
\end{equation}
\begin{equation}\label{example min 2.2}
    x_2 \geq 0;\;\;\;\; (\mathcal{A}_{min}x^2)_2 \geq 0;\;\;\;\; x_2[-x_2^2 +1] =0.
\end{equation}
Solving (\ref{example min 1.2}) and (\ref{example min 2.2}) we have SOL$(\Bar{q}, \mathcal{A}_{\min}) = \left\{ \left( \begin{array}{c}  0 \\ 0 \end{array} \right), \; \left( \begin{array}{c}  1 \\ 0 \end{array} \right), \; \left( \begin{array}{c}  0 \\ 1 \end{array} \right)\right\}.$ In this case we have $S^* = SOL(q_{\max}, \mathcal{A}_{\max}) \cap SOL(q_{\min}, \mathcal{A}_{\min}).$
\end{examp}

For the next result we define semi-infinite $S$-tensor.

\begin{defn}
    A tensor $\mathcal{A}(\omega)\in T_{m,n}$ is said to be a semi-infinite $S$-tensor with respect to the set $\omega$ if $\exists$ a vector $x>0$ such that $\mathcal{A}(\omega)x^{m-1} >0, \; \forall \; \omega \in \Omega.$ 
\end{defn}

Here we establish a connection between the semi-infinite $S$-tensor and the feasibility of SITCP$(q(\omega), \mathcal{A}(\omega), \Omega).$

\begin{theorem}\label{theorem with feasibility}
Consider the SITCP$(q(\omega), \mathcal{A}(\omega), \Omega)$ where $\Omega$ is compact and all the elements of $\mathcal{A}(\omega)$ are continuous on $\Omega.$ Then $\mathcal{A}(\omega)$ is a semi-infinite $S$-tensor relative to $\Omega$ if and only if SITCP$(q(\omega), \mathcal{A}(\omega), \Omega)$ is feasible for all $q(\omega) \in C(\Omega)$ where $C(\Omega)$ denotes all continuous mapping on $\Omega.$
\end{theorem}
\begin{proof}
If part: Since $\mathcal{A}(\omega)$ is a semi-infinite $S$-tensor, $\exists$ $x>0$ such that $\mathcal{A}(\omega)x^{m-1} >0.$ Then $\exists$ a sufficiently small scalar $\lambda > 0$ such that $\mathcal{A}(\omega)x^{m-1} \geq \lambda e >0,$ for all $\omega \in \Omega$ where $e = (1, \cdots , 1)^T.$  Now choose $\alpha > 0$ with $\alpha e > -q_{\min}.$ Choosing $\bar{\alpha} = \left(\frac{\alpha}{\lambda}\right)^{\frac{1}{m-1}}>0$ we have $\Bar{\alpha} x >0.$ Also, $\mathcal{A}(\omega)(\Bar{\alpha}x)^{m-1}$=$ \frac{\alpha}{\lambda} \mathcal{A}(\omega)x^{m-1} \geq \alpha e >  - q_{\min}.$ Thus for $\Bar{\alpha} x >0$ we have $\mathcal{A}(\omega)(\Bar{\alpha}x)^{m-1} + q(\omega) \geq \mathcal{A}(\omega)(\Bar{\alpha}x)^{m-1} + q_{\min} >0.$ Therefore, $\bar{\alpha} x$ is a feasible point of SITCP$(q(\omega), \mathcal{A}(\omega), \Omega)$.

Only if part: Let $q(\omega) := \Tilde{q} < 0$ for all $\omega \in \Omega.$ Let the SITCP$(\Tilde{q},$ $\mathcal{A}(\omega), \Omega)$ be feasible. Then $\exists$ a vector $x \geq 0$ such that $\mathcal{A}(\omega)x^{m-1} + \Tilde{q} \geq 0 \implies \mathcal{A}(\omega)x^{m-1} \geq -\Tilde{q} >0$ for all $\omega \in \Omega$. Since $\mathcal{A}(\omega)x^{m-1}$ is continuous on $\Omega$ there exists a sufficiently small $\lambda > 0$ such that $x + \lambda e > 0,$ and $\mathcal{A}(\omega)(x + \lambda e)^{m-1} >0 .$
\end{proof}

The following corollary provides a necessary condition for feasibility of the solution set of SITCP$(q(\omega), \mathcal{A}(\omega), \Omega).$

\begin{corol}
Consider the SITCP$(q(\omega), \mathcal{A}(\omega), \Omega).$ Suppose $\Omega$ is compact and $\mathcal{A}(\omega)$ is continuous on $\Omega.$ If $\mathcal{A}_{\min}$ is an $S$-tensor, then SITCP$(q(\omega), \mathcal{A}(\omega), \Omega)$ is feasible for all $q(\omega) \in C(\Omega).$
\end{corol}
\begin{proof}
Let $\mathcal{A}_{\min}$ be an $S$-tensor. Then for some $x>0$ we have $\mathcal{A}_{\min}x^{m-1} >0.$ From the definition of $\mathcal{A}_{\min}$ it follows that $\mathcal{A}(\omega)x^{m-1} > 0$ for all $\omega \in \Omega.$ This implies that $\mathcal{A}(\omega)$ is a semi-infinite $S$-tensor. Hence by Theorem \ref{theorem with feasibility}, we conclude that SITCP$(q(\omega), \mathcal{A}(\omega), \Omega)$ is feasible for all $q(\omega) \in C(\Omega).$
\end{proof}

Now we establish an $\epsilon-$error bound for the solution set of semi-infnte tensor complementarity problem. Here $\epsilon$ represents the degree of approximation of the set $S^*$

\begin{theorem}
Consider the SITCP$(q(\omega), \mathcal{A}(\omega), \Omega).$ Suppose the solution set $S^*$ is nonempty. If $\mathcal{A}(\omega_0)$ is an $P$-tensor for some $\omega_0 \in \Omega,$ then there exist $c > 0$ and $\epsilon > 0$ with $\epsilon \leq \mbox{diam} (\mbox{SOL}(q(\omega_0), \mathcal{A}(\omega_0))$ such that
\[ \mbox{dist}(x, S^*) \leq c\; r(x) + \epsilon \]
where residual function is $r_y(x) = \max_{\omega \Omega} \| \min \{ x, [\mathcal{A}(\omega)(x-y)]^{\frac{1}{m-1}} + [\mathcal{A}(\omega)y^{m-1}]^{\frac{1}{m-1}} \}\|,$ and $y\in (SOL(q(\omega_0), \mathcal{A}(\omega_0)).$ 
\end{theorem}
\begin{proof}
Since $\mathcal{A}(\omega_0)$ is an $P$-tensor, SOL$(q(\omega_0), \mathcal{A}(\omega_0))$ is bounded. This implies $\exists$ an $\epsilon > 0$ such that SOL$(q(\omega_0), \mathcal{A}(\omega_0)) \subseteq S^* + \epsilon \mathbb{B}.$ Consequently,
\begin{equation}\label{error bound 1.1}
   \mbox{dist}(x, S^*) \leq \mbox{dist}(x, SOL(q(\omega_0), \mathcal{A}(\omega_0))) + \epsilon, \forall x \in \mathbb{R}^n.
\end{equation}
Notice that $S^* \subseteq $ SOL$(q(\omega_0), \mathcal{A}(\omega_0)).$ Therefore the diameter of the set SOL$(q(\omega_0), \mathcal{A}(\omega_0))$ is an upper bound of $\epsilon.$ By Theorem 3.2 of \cite{zheng2019global} for the tensor complementarity problem TCP$(q(\omega_0), \mathcal{A}(\omega_0)),$ there exists $c > 0$ such that $\forall\; x \in \mathbb{R}^n,$
\begin{equation}\label{error bound 1.2}
    \mbox{dist}(x, SOL(q(\omega_0), \mathcal{A}(\omega_0))) \leq c \| \min \{ x, [\mathcal{A}(\omega_0)(x-y)]^{\frac{1}{m-1}} + [\mathcal{A}(\omega_0)y^{m-1}]^{\frac{1}{m-1}} \}\|.
\end{equation}
From (\ref{error bound 1.1}) and (\ref{error bound 1.2}) the desired result follows.
\end{proof}

For the next result we define semi-infinite $R_0$-tensor and establish a connection between $R_0$-tensor and semi-infinite $R_0$-tensor.

\begin{defn}
     The tensor $\mathcal{A}(\omega)$ is said to be a semi-infinite $R_0$-tensor relative to a set $\Omega$ if the SITCP$(0,\mathcal{A}(\omega), \Omega)$ has zero as its unique solution, i.e.,
     \begin{equation}
         x \geq 0,\;\; \mathcal{A}(\omega)x \geq 0,\;\;  x^T \mathcal{A}(\omega) x = 0,\;\;\;\; \forall \omega \in \Omega \implies x = 0. 
     \end{equation}
\end{defn}

\begin{theorem}
    Consider the SITCP$(q(\omega), \mathcal{A}(\omega), \Omega).$ If $\mathcal{A}(\omega_0)$ is an $R_0$-tensor for some $\omega_0 \in \Omega.$ Then $\mathcal{A}(\omega)$ is a semi-infinite $R_0$-tensor relative to $\Omega.$
\end{theorem}
\begin{proof}
     Since $\mathcal{A}(\omega_0)$ is an $R_0$-tensor, we have SOL$(0,\mathcal{A}(\omega_0))= \{0\}.$ Also, $S^* \subseteq \cap_{\omega \in \Omega}$ SOL$(0,\mathcal{A}(\omega), \Omega)$ Therefore $S^* = \{ 0\}.$ Hence the result.
\end{proof}

Here we prove the necessary and sufficient conditions for the error bound of solution set to be level bounded in terms of residual function. To solve the semi-infinite tensor complementarity problem is equivalent to finding $x\in \mathbb{R}^n$ such that $x\in \mbox{SOL}(q(\omega), \mathcal{A}(\omega))$ for all $\omega \in \Omega.$ However, in many cases, it is possible to obtain $x \in \mbox{SOL}(q(\omega), \mathcal{A}(\omega))$ for some $\omega.$ In this case, it is important to provide a quantitative measure of the closeness of each $x\in \mathbb{R}^n$ to each individual set SOL$(q(\omega), \mathcal{A}(\omega))$ in terms of some residual functions $r(x)$. In other words, we find $c > 0$ such that
\begin{equation*}
    \mbox{dist}(x, \mbox{SOL}(q(\omega), \mathcal{A}(\omega))) \leq c r(x), \; \forall \; \omega\in \Omega, \; \forall\; x\in \mathbb{R}^n
\end{equation*}
which is equivalent to
\begin{equation*}
    \max_{\omega \in \Omega} \mbox{dist}(x, \mbox{SOL}(q(\omega), \mathcal{A}(\omega))) \leq c r(x), \;\; \forall \; x\in \mathbb{R}^n.
\end{equation*}
Here $c$ is said to be weak error bound. The importance of weak error bound is that the solution $S^*$ is need not be nonempty as required in case of error bound.

\begin{theorem}
Consider the SITCP$(q(\omega), \mathcal{A}(\omega), \Omega).$ Suppose $\Omega$ is compact and $\mathcal{A}(\omega)$ and $q(\omega)$ are continuous. Then the residual function, $r(x) = \max_{\omega \in \Omega} \| \min(x, \mathcal{A}(\omega)x + q(\omega))\|^2$ is level bounded
if and only if the tensor $\mathcal{A}(\omega)$ is a semi-infinite $R_0$-tensor relative to $\Omega.$
\end{theorem}
\begin{proof}
Only if part: We prove the first result by contrapositive method. By this approach we first assume that $r(x)$ is not level bounded. Then $\exists$ a sequence $\{x_n\} \mapsto \infty$ as $n \mapsto \infty,$ $\{r(x_n)\}$ is bounded. We assume that $\frac{x_n}{\|x_n\|^{m-1}}$ converge to the limit $x_0$ with $\|x_0\| = 1.$ Taking into account the continuity of $q(\omega)$ and $\mathcal{A}(\omega)$ and the compactness of $\Omega$, we see that $r(x)$ is continuous and $q(\omega)$ is bounded on $\Omega.$ Hence, $\lim_{n\mapsto \infty} \frac{r(x_n)}{\|x_n\|^{m-1}}=0$ and $\lim_{n\mapsto \infty} \frac{q(\omega)}{\|x_n\|^{m-1}}=0$ for all $\omega \in \Omega.$ Now
\begin{equation}\label{equation for weak error bound  1st part}
    \frac{r(x_n)}{\|x_n\|^{2(m-1)}} = \max_{\omega \in \Omega} \| \min\left\{ \frac{x_n}{\|x_n\|^{m-1}}, \frac{\mathcal{A}(\omega)x + q(\omega)}{\|x_n\|^{m-1}} \right\} \|^2.
\end{equation}
Taking limit of both sides of the equation (\ref{equation for weak error bound  1st part}) as $n\mapsto \infty$ we obtain,
\[ \max_{\omega \in \Omega} \| \min\left\{x_0, \mathcal{A}(\omega)x_0^{m-1} \right\} \|^2 =0. \]
This means that a nonzero vector $x_0$ is a solution of SITCP$(0,\mathcal{A}(\omega), \Omega).$ Hence $\mathcal{A}(\omega)$ is not a semi-infinite $R_0$-tensor. This completes the proof.

If part: Suppose on the contrary that the SITCP$(0,\mathcal{A}(\omega), \Omega)$ has a nonzero vector $x$ as a solution. Let $I(x) = \{i: x_i = 0\}$ and $J(x) = \{i: x_i > 0\}.$ The compactness of $\Omega$ and the continuity of $q$ ensures that $q(\omega)$ is bounded on $\Omega.$ Thus there exists a scalar $K > 0$ such that, for any $k\geq K,$
\begin{equation}\label{last theorem 0}
    kx_i \geq q_i(\omega) \mbox{ for all } \omega \in \Omega \mbox{ and } i\in J(x).
\end{equation}
Given any $k \geq K,$ we have
\begin{align}\label{last theorem 1}
    r(kx) & = \max_{\omega \in \Omega} \| \min (kx, k^{m-1}\mathcal{A}(\omega)x^{m-1} + q(\omega))\|^2\\ \notag
          & \leq \sum_{i=1}^n \max_{\omega \in \Omega} \{ \min (k x_i, (k^{m-1} \mathcal{A}(\omega) x^{m-1})_i + q_i(\omega))\}^2
\end{align}
Now we consider the following cases.

\NI Case-1. We consider the case when $i \in J(x).$ Then $((\mathcal{A}(\omega)x^{m-1})_i = 0.$ It follows from (\ref{last theorem 0}) that
\begin{equation}\label{last theorem 2}
    \max_{\omega \in \Omega} \{ \min \{k x_i, k^{m-1} (\mathcal{A}(\omega)x^{m-1})_i + q_i(\omega))\}\}^2 = \max_{\omega \in \Omega}q_i(\omega)^2.
\end{equation}

\NI Case-2. We consider the case when $i \in I(x).$ If $ k^{m-1} (\mathcal{A}(\omega)x^{m-1})_i + q_i(\omega) \geq 0,$ we have
\begin{equation}\label{subcase 1}
    \{\min(k x_i, k^{m-1}(\mathcal{A}(\omega)x^{m-1})_i + q_i(\omega))\}^2 = 0.
\end{equation}
If $k^{m-1} (\mathcal{A}(\omega)x^{m-1})_i + q_i(\omega) < 0,$ then by the fact $q_i(\omega) \leq k^{m-1}(\mathcal{A}(\omega)x^{m-1})_i + q_i(\omega) < 0$ we obtain 
\begin{equation}\label{subcase 2}
    \{\min(k x_i, k^{m-1}(\mathcal{A}(\omega)x^{m-1})_i + q_i(\omega))\}^2 \leq q_i(\omega)^2.
\end{equation}
Thus combining (\ref{subcase 1}) and (\ref{subcase 2}) we have
\begin{equation}\label{last theorem 3}
    \max_{\omega \in \Omega} [\min(k x_i, k^{m-1} (\mathcal{A}(\omega)x^{m-1})_i + q_i(\omega))]^2 \leq \max_{\omega \in \Omega} q_i(\omega)^2.
\end{equation}
Putting the facts (\ref{last theorem 1}), (\ref{last theorem 2}) and (\ref{last theorem 3}) together, it follows that
\[ r(kx) \leq \sum_{i=1}^n \max_{\omega \in \Omega} q_i(\omega)^2 < \infty \]
for all $k \geq K.$ This contradicts the level boundedness of $r(x).$
\end{proof}

\section{Conclusion}
In this paper we introduce semi-infinite tensor complementarity problem to accommodate more realistic situation of the problem. We show that the solution set of semi-infinite tensor complementarity problem exists with some assumption. An example is illustrated in detail to establish the result. We establish a connection between semi-infinite tensor and its equivalent tensors to obtain the solution of semi-infinite tensor complementarity problem. Finally, we show that the error bound of the solution set is level bounded in terms of residual function.

\section{Acknowledgment}
The author R. Deb is grateful to the Council of Scientific $\&$ Industrial Research (CSIR), India, for providing financial assistance through the Junior Research Fellowship program.

\bibliographystyle{plain}
\bibliography{referencesall}

\end{document}